\documentclass{amsart}

\usepackage[utf8]{inputenc}
\usepackage{amssymb,latexsym,enumerate,diagrams,color}

\newtheorem{theorem}{Theorem}[section]
\newtheorem{lemma}[theorem]{Lemma}
\newtheorem{corollary}[theorem]{Corollary}

\newcommand{\Spec}{\mathop{\rm Spec}}
\newcommand{\id}{\mathord{\rm Id}}

\newcommand{\unit}{{\mathord{\times}}}

\newcommand{\im}{\operatorname{im}}
\newcommand{\coker}{\operatorname{coker}}

\let\ideal\mathfrak
\let\tensor\otimes

\let\sheaf\mathcal

\begin{document}

\title{A proof of Grothendieck's base change theorem}

\author{E. Tengan}

\address{Department of Mathematics, ICMC, University of S\~{a}o Paulo,
  13566-590, Brazil}

\email{etengan@icmc.usp.br}

\begin{abstract}
  We give an elementary short proof of Grothendieck's base change
  theorem for the cohomology of flat coherent sheaves.
\end{abstract}



\maketitle


\section{Introduction}

The purpose of this note is to give a short alternative proof of

\begin{theorem}[Grothendieck]\label{thm:main}
  Let $f\colon X \to Y$ be a proper map of noetherian schemes, and
  $\sheaf{F}$ be a coherent sheaf on $X$ which is flat over $Y$.  For
  $y\in Y$ let $X_y = X \times_Y \Spec \kappa(y)$ be the fiber of $y$,
  and $\sheaf{F}_y$ be the pullback of $\sheaf{F}$ to $X_y$.
  \begin{enumerate}[(a)]
  \item The base change map
    $$\varphi^p(y)\colon R^pf_* \sheaf{F} \tensor_{\sheaf{O}_Y}
    \kappa (y) \to H^p(X_y, \sheaf{F}_y)
    $$
    is surjective if and only if it is an isomorphism.

  \item Suppose that $\varphi^p(y)$ is surjective.  Then the following
    conditions are equivalent:
    \begin{enumerate}[(i)]
    \item $\varphi^{p-1}(y)$ is also surjective;
    \item $R^pf_* \sheaf{F}$ is a free sheaf in a neighborhood of $y$.
    \end{enumerate}
    Furthermore, if these conditions hold for all $y\in Y$, then the
    formation of $R^pf_* \sheaf{F}$ commutes with arbitrary base
    change.
  \end{enumerate}
\end{theorem}

The traditional proofs found in \cite{Hartshorne} (theorem 12.11,
p.290) and \cite{EGAIIIb} (théorème 7.7.5, p.67, proposition 7.7.10,
p.71, proposition 7.8.4, p.73) rely on either the formal functions
theorem or completion methods (in the spirit of the proof of the local
criterion of flatness).  On the other hand, Mumford \cite{Mumford}
(\S5, p.46) has given streamlined proofs of all the main results in
cohomology of base change except for above one.  Mumford's methods can
be readily adapted to prove theorem~\ref{thm:main} as well, and in a
quite elementary fashion.  Surprisingly, I could not find any written
account thereof; that is the reason why I decided to write one.


\section{Linear algebra over local rings}

All proofs of results in cohomology of base change are based on the
following key technical result (see \cite{Mumford}, \S 5, p.46):


\begin{theorem}[The Grothendieck complex\footnote{No, this is not any disorder suffered by those who unsuccesfully tried to learn
    scheme theory}]\label{thm:grothendieck}
  Let $f\colon X \to \Spec A$ be a proper map of noetherian schemes,
  and $\sheaf{F}$ be a coherent sheaf on $X$ which is $A$-flat.  There
  exists a finite complex $(F^\bullet, d^\bullet)$ of finitely
  generated $A$-flat modules such that for any $A$-algebra $B$ we have
  an isomorphism, functorial in $B$,
  $$H^p(X \tensor_A B, \sheaf{F} \tensor_A B) = H^p(F^\bullet \tensor_A B)
  $$
\end{theorem}

To prove theorem~\ref{thm:main}, since the question is local on $Y$
and all schemes involved are noetherian, we may assume that $Y = \Spec
A$ where $(A, \ideal{m}, k)$ is a local noetherian ring.  In
particular, the modules $F^p$ in the Grothendieck complex will all be
free of finite rank (recall that over a noetherian local ring, a
finitely generated module is flat if and only if it is projective if
and only if it is free, see \cite{Matsumura}, theorem 7.12, p.52).  We
may now write the base change map as
$$\varphi^p = \varphi^p(\ideal{m}) \colon H^p(F^\bullet)\tensor_A k \to H^p(F^\bullet\tensor_A k)$$
which reduces everything to proving the ``linear algebra'' lemmas
below.

We first make a simple remark regarding bases of a free module $F$ of
finite rank $n$ over a local ring $(A, \ideal{m}, k)$.  Denoting by a
bar the reduction modulo $\ideal{m}$,
$$e_1, \ldots, e_n \in F \text{ is an $A$-basis}
\iff \overline e_1, \ldots, \overline e_n \in
   F\tensor_A k = F/\ideal{m}F \text{ is a $k$-basis}
$$
In fact, clearly $F = \bigoplus_{1\le i\le n} Ae_i \implies F \tensor
k = \bigoplus_{1\le i\le n} k \overline e_i$; conversely, writing the
$e_i = \sum_{1\le j\le n} a_{ij} f_j$ ($a_{ij} \in A$) in terms of an
$A$-basis $f_j$ of $F$, we have that the matrix $(\overline a_{ij})$
is invertible since both the $\overline e_i$ and the $\overline f_i$
form $k$-bases of $F/\ideal{m}F$, hence $\det (\overline a_{ij})\ne
\overline 0 \in k \iff \det (a_{ij}) \in A^\unit$, showing that
$(a_{ij})$ is also invertible, and therefore the $e_i$ also form an
$A$-basis of $F$.

\begin{lemma}
  Let $(A, \ideal{m}, k)$ be a local ring, and $d\colon F \to F'$ be a
  map of free $A$-modules of finite rank.  There exist decompositions
  $$F = V \oplus W \qquad F' = W' \oplus U$$
  such that $d(V) \subseteq \ideal{m} F'$ and $d$ restricts to an
  isomorphism $d\colon W \stackrel{\approx}{\to} W'$.  In other words,
  there exist $A$-bases of $F$ and $F'$ with respect to which
  $$d = \begin{pmatrix}
    M_{s\times r} & \id_{s\times s}\\
    N_{t\times r} & 0_{t\times s}
  \end{pmatrix}
  $$
  with all entries of $M$, $N$ belonging to $\ideal{m}$.
\end{lemma}

\begin{proof}
  Consider the $k$-linear map $d\tensor 1\colon F\tensor_A k \to
  F'\tensor_A k$, and choose $e_1, \ldots, e_{r+s}\in F$ so that
  $\overline e_1, \ldots, \overline e_{r+s}$ is a $k$-basis of
  $F\tensor_A k$, with the first $r$ vectors generating $\ker
  (d\tensor 1)$.  Hence $\overline{d(e_{r+1})}, \ldots,
  \overline{d(e_{r+s})} \in F'\tensor_A k = F'/\ideal{m}F'$ is a
  $k$-basis of $\im (d\tensor 1)$.  By the above remark, the $e_i$ form an
  $A$-basis of $F$, and we may find an $A$-basis $f_1, \ldots,
  f_{s+t}$ of $F'$ with $f_1 = d(e_{r+1}), \ldots, f_s = d(e_{r+s})$.
  Now set
  \begin{align*}
    V &= A e_1 \oplus \cdots \oplus A e_r&
    W &= A e_{r+1} \oplus \cdots \oplus A e_{r+s}\\
    W' &= A f_1 \oplus \cdots \oplus A f_s&
    U &= A f_{s+1} \oplus \cdots \oplus A f_{s+t}
  \end{align*}
  and we are done.
\end{proof}

\goodbreak

\begin{lemma}
  In the notation above, the following conditions are equivalent:
  \begin{enumerate}[(i)]
  \item $\varphi^p$ is an isomorphism;
  \item $\varphi^p$ is surjective;
  \item $d^p$ can be put in matrix form
    $$d^p = \begin{pmatrix}
      0& \id\\
      0& 0
    \end{pmatrix}
    $$
    for some choice of $A$-bases of $F^p$ and $F^{p+1}$.
  \item $\ker d^p$ and $\im d^p$ are direct summands of $F^p$ and
    $F^{p+1}$, respectively (in particular, they are free since $A$ is
    local noetherian).
  \end{enumerate}
\end{lemma}

\begin{proof}
  Clearly $(i) \Rightarrow (ii)$ and $(iii) \Leftrightarrow (iv)$.
  Next, observe that $(F^\bullet\tensor_A k, d^\bullet \tensor 1)$ can
  be written as
  $$\cdots \rTo \frac{F^{p-1}}{\ideal{m} F^{p-1}} \rTo^{\overline d^{p-1}}
  \frac{F^{p}}{\ideal{m} F^{p}}\rTo^{\overline d^{p}}
  \frac{F^{p+1}}{\ideal{m} F^{p+1}}\rTo \cdots
  $$
  and the base change map $\varphi^p \colon H^p(F^\bullet)\tensor_A k
  \to H^p(F^\bullet\tensor_A k)$ as the natural map
  $$\varphi^p\colon \frac{\ker d^p}{\im d^{p-1} + \ideal{m} \ker d^p}
  \to \frac{(d^p)^{-1} (\ideal{m} F^{p+1})}{\im d^{p-1} + \ideal{m} F^p}
  \eqno{(*)}
  $$
  This shows that $(iii) \Rightarrow (i)$: if $F^p = V \oplus W$ is
  the corresponding decomposition in $(iii)$ with $V = \ker d^p$,
  we have
  $$\frac{(d^p)^{-1} (\ideal{m} F^{p+1})}{\im d^{p-1} + \ideal{m} F^p}
  = \frac{V \oplus \ideal{m}W}{\im d^{p-1} + \ideal{m}V \oplus \ideal{m}W}
  = \frac{V}{\im d^{p-1} + \ideal{m} V}
  $$
  since $\im d^{p-1} \subseteq V = \ker d^p$, and thus $\varphi^p$ is
  an isomorphism.

  Finally, to prove that $(ii)\Rightarrow (iii)$, notice first that
  from $(*)$ we get
  $$\varphi^p \text{ is surjective} \iff \ker d^p + \ideal{m} F^p =
  (d^p)^{-1} (\ideal{m} F^{p+1})
  \eqno{(**)}
  $$
  Now applying the previous lemma to $d^p$, there are decompositions
  $F^p = V \oplus W$ and $F^{p+1} = W' \oplus U$ with respect to which
  $d^p$ has matrix
  $$d^p = \begin{pmatrix}
    M & \id\\
    N & 0
    \end{pmatrix}
  $$
  where all entries of $M$ and $N$ belong to $\ideal{m}$.  Therefore
  $$(d^p)^{-1} (\ideal{m} F^{p+1}) = V \oplus \ideal{m}W
  \qquad\text{and}\qquad
  \ker d^p = \{ (v, -Mv) \in V\oplus W \mid Nv = 0 \}
  $$
  and the right hand side of $(**)$ becomes
  $$\ker N + \ideal{m}V = V  \stackrel{\text{Nakayama}}{\iff} \ker N = V
  \iff N = 0
  $$
  But now $\ker d^p = \{ (v, -Mv) \in V \oplus W \mid v\in V\}$ is a
  free summand of $F^p$, and a final change of basis puts $d^p$ in the
  desired format: just right multiply it by the following invertible
  matrix, whose columns on the left form a basis of $\ker d^p$.
  $$\begin{pmatrix}
    \id & 0\\
    -M & \id
    \end{pmatrix}
  $$
\end{proof}

\begin{lemma}
  Suppose that $\varphi^p$ is surjective.  Then the following are
  equivalent:
  \begin{enumerate}[(i)]
  \item $\varphi^{p-1}$ is surjective;
  \item $H^p(F^\bullet)$ is free.
  \end{enumerate}
  Furthermore, if these conditions hold, then $H^p(F^\bullet)
  \tensor_A B = H^p(F^\bullet\tensor_A B)$ for any $A$-algebra $B$.
\end{lemma}

\begin{proof}
  Since $\varphi^p$ is surjective, by the previous lemma $\ker d^p$ is
  free, hence replacing $F^p$ by $\ker d^p$ we may assume that $d^{p}
  = 0$, and we have exact sequences
  \begin{align*}
    0&\rTo \im d^{p-1} \rTo F^p \rTo H^p(F^\bullet)
    \rTo 0\\
    0&\rTo \ker d^{p-1} \rTo F^{p-1} \rTo^{d^{p-1}} \im d^{p-1} \rTo 0
  \end{align*}
  Now if $\varphi^{p-1}$ is surjective, then by the previous lemma
  $H^p(F^\bullet) = F^p/\im d^{p-1}$ is free. Conversely, suppose
  $H^p(F^\bullet)$ is free.  Then the first sequence splits, showing
  that $\im d^{p-1}$ is free, and thus the second sequence splits as
  well.  Therefore both $\ker d^{p-1}$ and $\im d^{p-1}$ are direct
  summands of $F^{p-1}$ and $F^p$ respectively, and by the previous
  lemma $\varphi^{p-1}$ is surjective.

  Finally, $H^p(F^\bullet) \tensor_A B = H^p(F^\bullet\tensor_A B)\iff
  (\coker d^{p-1}) \tensor_A B = \coker (d^{p-1} \tensor 1)$ also
  directly follows from the matrix form of $d^{p-1}$ in $(iii)$ of the
  previous lemma.
\end{proof}


\section{A useful corollary}

For completion, we include one of the main applications of
theorem~\ref{thm:main}, namely the following corollary, ``which is
extremely useful, but which is unfortunately buried there [in EGA III]
in a mass of generalizations'' (\cite{MumfordFogarty}, p.19).

\goodbreak

\begin{corollary}
  Let $f\colon X \to Y$ be a proper map of noetherian schemes, and
  $\sheaf{F}$ be a coherent sheaf on $X$ which is flat over $Y$.  If
  $H^1(X_y, \sheaf{F}_y) = 0$ for all $y\in Y$, then
  \begin{enumerate}[(a)]
  \item $R^1f_* \sheaf{F}  = 0$
  \item $f_*\sheaf{F}$ is a locally free $\sheaf{O}_Y$-module, whose
    formation commutes with arbitrary base change.
  \end{enumerate}
\end{corollary}

\begin{proof}
  By theorem~\ref{thm:main}(a) applied with $p=1$,
  $R^1f_*\sheaf{F}\tensor_{\sheaf{O}_Y} \kappa(y) = 0$ for all $y\in
  Y$, and since $R^1f_*\sheaf{F}$ is coherent (Serre's theorem, see
  \cite{EGAIIIa}, théorème 3.2.1, p.116), $R^1f_*\sheaf{F} = 0$ by
  Nakayama's lemma.  Now by theorem~\ref{thm:main}(b) applied with
  $p=1$, we have that $\varphi^0(y)$ is surjective for all $y\in Y$.
  Finally, applying theorem~\ref{thm:main}(b) again with $p=0$
  finishes the proof: $\varphi^{-1}(y)$ is an isomorphism since both
  $R^pf_* \sheaf{F} \tensor_{\sheaf{O}_Y} \kappa (y)$ and $H^p(X_y,
  \sheaf{F}_y)$ vanish for $p=-1$ (alternatively, it is easy to check
  that $H^0(F^\bullet)$ is free directly: since $R^1f_*\sheaf{F} = 0$,
  the Grothendieck complex is exact at $p=1$, and since $\varphi^1$ is
  an isomorphism, $\ker d^1 = \im d^0$ is a direct summand of $F^1$,
  hence free; now the exact sequence $0 \to \ker d^0 \to F^0 \to \im
  d^0 \to 0$ splits, showing that $\ker d^0 = H^0(F^\bullet)$ is
  free).
\end{proof}


\section{Acknowledgments}

This work was supported by CNPq grant 303817/2011-9.  I would also
like to take this opportunity to thank my friend Edmilson Motta for
not letting me forget Master Yoda's perennial wisdom: ``Do.  Or do
not.  There is no try.''

\bibliographystyle{alpha} 
\bibliography{BaseChangeBib}

\end{document}